\documentclass[12pt]{article}
\usepackage{amsmath, amsthm}
\usepackage{times}
\usepackage{newtxmath}
\usepackage[margin=1.0in]{geometry}

\newtheorem{theorem}{Theorem}
\newtheorem{remark}{Remark}

\theoremstyle{definition}
\newtheorem{definition}{Definition}
\newtheorem{reference}{Reference Summarization}

\begin{document}

\title{Generalized Motives through Witt Vectors}
\author{Xin Tong}
\date{}

\maketitle

\begin{abstract}
\noindent We generalize condensed motives, through Witt vectors and the associated moduli stacks of generalized line bundles.

\end{abstract}

\newpage
\tableofcontents

\newpage

\begin{reference}
We follow \cite{G} on the corresponding motives. We start from the foundation in \cite{BSI}, \cite{BLI}, \cite{DI}, \cite{SchI}, \cite{ALBRCS}, \cite{TI}, \cite{TII}, \cite{TIII}, \cite{TVI}. For condensed mathematics consideration see \cite{CS1}, \cite{CS2}, \cite{CS3}. This paper is also closely following and inspired by \cite{TIV} and \cite{TV}. We apply the generalization then to the level of consideration closely after \cite{Ta}, \cite{FI}, \cite{KLI}, \cite{KLII}, \cite{SchII}, \cite{SchIII}, \cite{SchIV}, as well as closely after \cite{L}, \cite{FS}, \cite{DHKM}, \cite{VL}, \cite{LL}, \cite{EGH}, \cite{DIII}, \cite{DII}, \cite{Z}, \cite{GL}. Witt vectors are very significant in this paper since all the corresponding prismatizations are defined to be the Cartier moduli stacks parametrizing those divisors living in the Witt vector spaces, in either perfect setting (see \cite{SchIII}, \cite{SchII}, \cite{KLI}, \cite{KLII} as well for more inspiration and development) and the imperfect setting. Prismatization will mean taking the spectrum (up to the Cartier primes) of the Witt vectors, while perfect prismatization relates to stacks of untilts such as in the work of Scholze and Kedlaya-Liu in \cite{SchIII}, \cite{SchII}, \cite{KLI}, \cite{KLII}. By Witt vectors we will mean the most generalized version such as in \cite{KLI}, \cite{KLII}, \cite{FS}, which allow us to consider the $z$-adic setting, for instance the Fontaine-Wintenberger parametrization of the Cartier divisors is considered extensively in the $z$-adic setting in \cite{KLI} and \cite{KLII}.
\end{reference}

\indent This paper generalizes motives. Motives are universalization of cohomologies after \cite{G}, therefore generalization of universal cohomology theory is a completely mysterious and sophisticated consideration. Extending the original idea on the motives can be very tricky. What we considered here is the so called mixed-parity filtred generalization of the usual motives after \cite{BSB}. Even in the situation where we do not consider Galois representations, this can be for instance realized by using generalized Langlands program. Representation of Galois groups can be regarded as certain sheaves over more generalized Hodge theoretic spaces and even fiber categories. For any $p$-adic formal scheme we generalize the corresponding prismatization from \cite{BSI}, \cite{BLI} and \cite{DI}. We also have the $z$-adic version as well the corrsponding $v$-stack version, where the idea is to consider certain generalized paramatrization stacks of Cartier sheaves, extensively by using the Witt vectors attached to the field $T$. We then apply this construction to rigid analytic space $R$ by using the $v$-stack consideration, which is completely compatible with \cite{TI}, \cite{TII}, \cite{TIII}, \cite{SchI}, \cite{ALBRCS}. This immediately allows us to generalize the $z$-adic cohomology theory after \cite{KLI}, \cite{KLII}, \cite{SchII}, \cite{SchIII}, \cite{SchIV}. Finally we consider the generalized Langlands program after \cite{TIV}, \cite{TV}, \cite{TVI}, \cite{L}, \cite{FS}, \cite{VL}, \cite{DII}, \cite{DIII}. See the final chapter.

\section{$p$-adic Motives}
\subsection{Generalized Prismatization}
We start from the foundation in \cite{BSI}, \cite{BLI}, \cite{DI}, \cite{SchI}, \cite{ALBRCS}, \cite{TI}, \cite{TII}, \cite{TIII}, \cite{TVI}. We start from the corresponding formal scheme situation. We will start from a finite extension $T$ of $\mathbb{Q}_p$. We consider a general $F$ which is a $p$-adic formal scheme over $\mathcal{O}_T$. Attached to $F$ we have the corresponding primatic stack:
\begin{align}
\mathrm{\Gamma}_{\mathrm{prismatization},F}
\end{align}  
which is defined actually over rings where $p$ is nilpotent. We use the notation:
\begin{align}
\mathrm{C}_{p\mathrm{nil}} 
\end{align}
to be this category. Therefore the stack
\begin{align}
\mathrm{\Gamma}_{\mathrm{prismatization},F}
\end{align} 
is then defined over this category as the family of Cartier divisors (\textit{closed immersingly}) living in the corresponding generalized Witt vector of any ring $H$ in the under category:
\begin{align}
?\rightarrow \mathrm{WVector}_{\mathcal{O}_T}(H).
\end{align}
This stack is actually formal stack. One then have the corresponding quasicoherent sheaves over this stack as the corresponding $\infty$-category of all the quasicoherent modules over the structure sheaf:
\begin{align}
\mathrm{Quasicoherent}_{\mathrm{\Gamma}_{\mathrm{prismatization},F}}
\end{align}
which can be presented as the inverse limit of certain $\infty$-category of all the quasicoherent modules over a family of prisms in coherent way:
\begin{align}
\mathrm{Quasicoherent}_{\mathrm{\Gamma}_{\mathrm{prismatization},F}}=\projlim_i \mathrm{Quasicoherent}_{\mathrm{Spec}U_i}
\end{align}
where $U_i$ is the desired family of such sort of prisms\footnote{In fact one can care more about the quasicoherent sheaves instead of the actual stacks. This point of view is quite well-known in the field of algebraic geometry where even in certain situation the underlying stack can be reconstructed from the corresponding category of quasicoherent sheaves.}. This presentation can be directly use to construct the corresponding de Rham stacks for instance:
\begin{align}
\mathrm{\Gamma}_{\mathrm{prismatization},\mathrm{deRham},F}:=\injlim_i \mathrm{Spec}U_i[1/p]_{I_{U_i}}
\end{align}
with the corresponding quasicoherent sheaves category:
\begin{align}
\mathrm{Quasicoherent}_{\mathrm{\Gamma}_{\mathrm{prismatization},\mathrm{deRham},F}}=\projlim_i \mathrm{Quasicoherent}_{\mathrm{Spec}U_i[1/p]_{I_{U_i}}}.
\end{align}
One then have the analytification version of this stack:
\begin{align}
\mathrm{\Gamma}^\sharp_{\mathrm{prismatization},\mathrm{deRham},F}:=\injlim_i \mathrm{Spec}^\sharp U_i[1/p]_{I_{U_i}}
\end{align}
with the corresponding quasicoherent sheaves category:
\begin{align}
\mathrm{Quasicoherent}^\sharp_{\mathrm{\Gamma}_{\mathrm{prismatization},\mathrm{deRham},F}}=\projlim_i \mathrm{Quasicoherent}^\sharp_{\mathrm{Spec}U_i[1/p]_{I_{U_i}}}.
\end{align}
We then have the corresponding cristalline version as well. This presentation can be directly use to construct the corresponding cristalline stacks for instance:
\begin{align}
\mathrm{\Gamma}_{\mathrm{prismatization},\mathrm{cristalline},F}:=\injlim_i \mathrm{Spec}U_i[1/p]_{I_{U_i}}
\end{align}
with the corresponding quasicoherent sheaves category:
\begin{align}
\mathrm{Quasicoherent}_{\mathrm{\Gamma}_{\mathrm{prismatization},\mathrm{cristalline},F}}=\projlim_i \mathrm{Quasicoherent}_{\mathrm{Spec}U_i[1/p]_{I_{U_i}}}.
\end{align}
One then have the analytification version of this stack:
\begin{align}
\mathrm{\Gamma}^\sharp_{\mathrm{prismatization},\mathrm{cristalline},F}:=\injlim_i \mathrm{Spec}^\sharp U_i[1/p]_{I_{U_i}}
\end{align}
with the corresponding quasicoherent sheaves category:
\begin{align}
\mathrm{Quasicoherent}^\sharp_{\mathrm{\Gamma}_{\mathrm{prismatization},\mathrm{cristalline},F}}=\projlim_i \mathrm{Quasicoherent}^\sharp_{\mathrm{Spec}U_i[1/p]_{I_{U_i}}}.
\end{align} 

\begin{definition}
We now generalize the corresponding foundation above to \cite{BSB} in the following sense. First starting with the stack:
\begin{align}
\mathrm{\Gamma}_{\mathrm{prismatization},F}
\end{align}
we have the corresponding $p$-adic primitive element $h$ from \cite{BSB} which is usual the corresponding logrithmic with respect to the imperfect Robba ring (in variable $X$, then it will be the log of $1+X$) attached to the field $T$. We then add $h^{1/2}$ to the following map in the definition of the primatization:
\begin{align}
?\rightarrow \mathrm{WVector}_{\mathcal{O}_T}(H)[h^{1/2}].
\end{align}
Then we have the corresponding stack which we will denote it by:
\begin{align}
\mathrm{\Gamma}_{\mathrm{prismatization},F,2}.
\end{align}
Then the de Rham and cristalline situations are generalized accordingly. This stack is actually formal stack. One then has the corresponding quasicoherent sheaves over this stack as the corresponding $\infty$-category of all the quasicoherent modules over the structure sheaf:
\begin{align}
\mathrm{Quasicoherent}_{\mathrm{\Gamma}_{\mathrm{prismatization},F,2}}
\end{align}
which can be presented as the inverse limit of certain $\infty$-category of all the quasicoherent modules over a family of prisms in coherent way:
\begin{align}
\mathrm{Quasicoherent}_{\mathrm{\Gamma}_{\mathrm{prismatization},F,2}}=\projlim_i \mathrm{Quasicoherent}_{\mathrm{Spec}U_i[h^{1/2}]}
\end{align}
where $U_i$ is the desired family of such sort of prisms. This presentation can be directly use to construct the corresponding de Rham stacks for instance:
\begin{align}
\mathrm{\Gamma}_{\mathrm{prismatization},\mathrm{deRham},F,2}:=\injlim_i \mathrm{Spec}U_i[1/p]_{I_{U_i}}[h^{1/2}]
\end{align}
with the corresponding quasicoherent sheaves category:
\begin{align}
\mathrm{Quasicoherent}_{\mathrm{\Gamma}_{\mathrm{prismatization},\mathrm{deRham},F,2}}=\projlim_i \mathrm{Quasicoherent}_{\mathrm{Spec}U_i[1/p]_{I_{U_i}}[h^{1/2}]}.
\end{align}
One then have the analytification version of this stack:
\begin{align}
\mathrm{\Gamma}^\sharp_{\mathrm{prismatization},\mathrm{deRham},F,2}:=\injlim_i \mathrm{Spec}^\sharp U_i[1/p]_{I_{U_i}}[h^{1/2}]
\end{align}
with the corresponding quasicoherent sheaves category:
\begin{align}
\mathrm{Quasicoherent}^\sharp_{\mathrm{\Gamma}_{\mathrm{prismatization},\mathrm{deRham},F,2}}=\projlim_i \mathrm{Quasicoherent}^\sharp_{\mathrm{Spec}U_i[1/p]_{I_{U_i}}[h^{1/2}]}.
\end{align}
We then have the corresponding cristalline version as well. This presentation can be directly use to construct the corresponding cristalline stacks for instance:
\begin{align}
\mathrm{\Gamma}_{\mathrm{prismatization},\mathrm{cristalline},F,2}:=\injlim_i \mathrm{Spec}U_i[1/p][h^{1/2}]
\end{align}
with the corresponding quasicoherent sheaves category:
\begin{align}
\mathrm{Quasicoherent}_{\mathrm{\Gamma}_{\mathrm{prismatization},\mathrm{cristalline},F,2}}=\projlim_i \mathrm{Quasicoherent}_{\mathrm{Spec}U_i[1/p][h^{1/2}]}.
\end{align}
One then have the analytification version of this stack:
\begin{align}
\mathrm{\Gamma}^\sharp_{\mathrm{prismatization},\mathrm{cristalline},F,2}:=\injlim_i \mathrm{Spec}^\sharp U_i[1/p][h^{1/2}]
\end{align}
with the corresponding quasicoherent sheaves category:
\begin{align}
\mathrm{Quasicoherent}^\sharp_{\mathrm{\Gamma}_{\mathrm{prismatization},\mathrm{cristalline},F,2}}=\projlim_i \mathrm{Quasicoherent}^\sharp_{\mathrm{Spec}U_i[1/p][h^{1/2}]}.
\end{align} 
\end{definition}

\begin{theorem}
When $F$ is $\mathcal{O}_T$ we have a generalized condensed prismatization, which is well-defined.
\end{theorem}

\begin{remark}
We remark that in order for us to relate this generalization to Galois representations, one has to use Fargues-Fontaine curves. The reason behind this is on \textit{homotopicalization}. Recall that Fargues-Fontaine curves are encoding the Galois actions. Therefore over $\mathcal{C}_T$ we have the corresponding \textbf{fundamental groups} of FF curves as the corresponding Galois groups of $T$. Therefore in order to have action from two fold covering groups of the Galois groups of $T$ and even in certain relative fashion, one has to use the generalized FF curves by extending the action of the action from Galois groups of $T$ to the extension of these Galois groups, then we have a functor from the corresponding generalized prismatization to the corresponding $G_{T,2}$-equivariant sheaves over generalized FF curves. This parallels to $z$-adic situation as well.
\end{remark}

\subsection{Perfectoid Picture}
\noindent We then consider the contact with the work of Kedlaya-Liu and Scholze, where untilts of perfectoids are parametrized by certain similar stacks, i.e. the Fargues-Fontaine stacks, from \cite{KLI}, \cite{KLII}, \cite{SchII}, \cite{SchIII}, \cite{SchIV}. We now consider any small $v$-stack $S$ over $\mathrm{Spd}\mathcal{O}_T$. For any local chart ${F}$ in the $v$-topology over this stack $S$, we can get the corresponding prismatization directly in a transparent way, namely the corresponding prismatization is just the corresponding space of the Witt vector attached to $F$. However this is a prismatic method. Attached to $F$ we have the corresponding primatic stack:
\begin{align}
\mathrm{\Gamma}_{\mathrm{prismatization},F}
\end{align}  
which is defined actually over rings where $p$ is nilpotent. We use the notation:
\begin{align}
\mathrm{C}_{p\mathrm{nil}} 
\end{align}
to be this category. Therefore the stack
\begin{align}
\mathrm{\Gamma}_{\mathrm{prismatization},F}
\end{align} 
is then defined over this category as the family of Cartier divisors (\textit{closed immersingly}) living in the corresponding generalized Witt vector of any ring $H$ in the under category:
\begin{align}
?\rightarrow \mathrm{WVector}_{\mathcal{O}_T}(H).
\end{align}
This stack is actually formal stack. One then have the corresponding quasicoherent sheaves over this stack as the corresponding $\infty$-category of all the quasicoherent modules over the structure sheaf:
\begin{align}
\mathrm{Quasicoherent}_{\mathrm{\Gamma}_{\mathrm{prismatization},F}}
\end{align}
which can be presented as the inverse limit of certain $\infty$-category of all the quasicoherent modules over a family of prisms in coherent way:
\begin{align}
\mathrm{Quasicoherent}_{\mathrm{\Gamma}_{\mathrm{prismatization},F}}=\projlim_i \mathrm{Quasicoherent}_{\mathrm{Spec}U_i}
\end{align}
where $U_i$ is the desired family of such sort of prisms. This presentation can be directly use to construct the corresponding de Rham stacks for instance:
\begin{align}
\mathrm{\Gamma}_{\mathrm{prismatization},\mathrm{deRham},F}:=\injlim_i \mathrm{Spec}U_i[1/p]_{I_{U_i}}
\end{align}
with the corresponding quasicoherent sheaves category:
\begin{align}
\mathrm{Quasicoherent}_{\mathrm{\Gamma}_{\mathrm{prismatization},\mathrm{deRham},F}}=\projlim_i \mathrm{Quasicoherent}_{\mathrm{Spec}U_i[1/p]_{I_{U_i}}}.
\end{align}
One then have the analytification version of this stack:
\begin{align}
\mathrm{\Gamma}^\sharp_{\mathrm{prismatization},\mathrm{deRham},F}:=\injlim_i \mathrm{Spec}^\sharp U_i[1/p]_{I_{U_i}}
\end{align}
with the corresponding quasicoherent sheaves category:
\begin{align}
\mathrm{Quasicoherent}^\sharp_{\mathrm{\Gamma}_{\mathrm{prismatization},\mathrm{deRham},F}}=\projlim_i \mathrm{Quasicoherent}^\sharp_{\mathrm{Spec}U_i[1/p]_{I_{U_i}}}.
\end{align}
We then have the corresponding cristalline version as well. This presentation can be directly use to construct the corresponding cristalline stacks for instance:
\begin{align}
\mathrm{\Gamma}_{\mathrm{prismatization},\mathrm{cristalline},F}:=\injlim_i \mathrm{Spec}U_i[1/p]_{I_{U_i}}
\end{align}
with the corresponding quasicoherent sheaves category:
\begin{align}
\mathrm{Quasicoherent}_{\mathrm{\Gamma}_{\mathrm{prismatization},\mathrm{cristalline},F}}=\projlim_i \mathrm{Quasicoherent}_{\mathrm{Spec}U_i[1/p]_{I_{U_i}}}.
\end{align}
One then have the analytification version of this stack:
\begin{align}
\mathrm{\Gamma}^\sharp_{\mathrm{prismatization},\mathrm{cristalline},F}:=\injlim_i \mathrm{Spec}^\sharp U_i[1/p]_{I_{U_i}}
\end{align}
with the corresponding quasicoherent sheaves category:
\begin{align}
\mathrm{Quasicoherent}^\sharp_{\mathrm{\Gamma}_{\mathrm{prismatization},\mathrm{cristalline},F}}=\projlim_i \mathrm{Quasicoherent}^\sharp_{\mathrm{Spec}U_i[1/p]_{I_{U_i}}}.
\end{align} 

\begin{definition}
We now generlize the corresponding foundation above to \cite{BSB} in the following sense. First starting with the stack:
\begin{align}
\mathrm{\Gamma}_{\mathrm{prismatization},F}
\end{align}
we have the corresponding $p$-adic primitive element $h$ from \cite{BSB} which is usual the corresponding logrithmic with respect to the imperfect Robba ring (in variable $X$, then it will be the log of $1+X$) attached to the field $T$. We then add $h^{1/2}$ to the following map in the definition of the primatization:
\begin{align}
?\rightarrow \mathrm{WVector}_{\mathcal{O}_T}(H)[h^{1/2}].
\end{align}
Then we have the corresponding stack which we will denote it by:
\begin{align}
\mathrm{\Gamma}_{\mathrm{prismatization},F,2}.
\end{align}
Then the de Rham and cristalline situations are generalized accordingly. This stack is actually formal stack. One then have the corresponding quasicoherent sheaves over this stack as the corresponding $\infty$-category of all the quasicoherent modules over the structure sheaf:
\begin{align}
\mathrm{Quasicoherent}_{\mathrm{\Gamma}_{\mathrm{prismatization},F,2}}
\end{align}
which can be presented as the inverse limit of certain $\infty$-category of all the quasicoherent modules over a family of prisms in coherent way:
\begin{align}
\mathrm{Quasicoherent}_{\mathrm{\Gamma}_{\mathrm{prismatization},F,2}}=\projlim_i \mathrm{Quasicoherent}_{\mathrm{Spec}U_i[h^{1/2}]}
\end{align}
where $U_i$ is the desired family of such sort of prisms. This presentation can be directly use to construct the corresponding de Rham stacks for instance:
\begin{align}
\mathrm{\Gamma}_{\mathrm{prismatization},\mathrm{deRham},F,2}:=\injlim_i \mathrm{Spec}U_i[1/p]_{I_{U_i}}[h^{1/2}]
\end{align}
with the corresponding quasicoherent sheaves category:
\begin{align}
\mathrm{Quasicoherent}_{\mathrm{\Gamma}_{\mathrm{prismatization},\mathrm{deRham},F,2}}=\projlim_i \mathrm{Quasicoherent}_{\mathrm{Spec}U_i[1/p]_{I_{U_i}}[h^{1/2}]}.
\end{align}
One then have the analytification version of this stack:
\begin{align}
\mathrm{\Gamma}^\sharp_{\mathrm{prismatization},\mathrm{deRham},F,2}:=\injlim_i \mathrm{Spec}^\sharp U_i[1/p]_{I_{U_i}}[h^{1/2}]
\end{align}
with the corresponding quasicoherent sheaves category:
\begin{align}
\mathrm{Quasicoherent}^\sharp_{\mathrm{\Gamma}_{\mathrm{prismatization},\mathrm{deRham},F,2}}=\projlim_i \mathrm{Quasicoherent}^\sharp_{\mathrm{Spec}U_i[1/p]_{I_{U_i}}[h^{1/2}]}.
\end{align}
We then have the corresponding cristalline version as well. This presentation can be directly use to construct the corresponding cristalline stacks for instance:
\begin{align}
\mathrm{\Gamma}_{\mathrm{prismatization},\mathrm{cristalline},F,2}:=\injlim_i \mathrm{Spec}U_i[1/p][h^{1/2}]
\end{align}
with the corresponding quasicoherent sheaves category:
\begin{align}
\mathrm{Quasicoherent}_{\mathrm{\Gamma}_{\mathrm{prismatization},\mathrm{cristalline},F,2}}=\projlim_i \mathrm{Quasicoherent}_{\mathrm{Spec}U_i[1/p][h^{1/2}]}.
\end{align}
One then have the analytification version of this stack:
\begin{align}
\mathrm{\Gamma}^\sharp_{\mathrm{prismatization},\mathrm{cristalline},F,2}:=\injlim_i \mathrm{Spec}^\sharp U_i[1/p][h^{1/2}]
\end{align}
with the corresponding quasicoherent sheaves category:
\begin{align}
\mathrm{Quasicoherent}^\sharp_{\mathrm{\Gamma}_{\mathrm{prismatization},\mathrm{cristalline},F,2}}=\projlim_i \mathrm{Quasicoherent}^\sharp_{\mathrm{Spec}U_i[1/p][h^{1/2}]}.
\end{align} 
\end{definition}

\indent Then we vary $F$ in the $v$-site for $S$ we then have the corresponding the condensed prismatization of the $v$-stack $S$ in the mixed-parity generalization fashion.

\newpage
\section{$z$-adic Motives}

\subsection{Generalized Prismatization}
Since there is $z$-adic version of Fargues-Fontaine stacks, there \textit{should} be at least\footnote{Over real number $\mathbb{R}$ or complex number $\mathbb{C}$, one can also consider the \textit{prismatization} by using the moduli stacks of line bundles in the Witt vectors, at least by deformation to:
\begin{align}
\mathbb{R}((u)), \mathbb{C}((u)),
\end{align}
namely the \textit{nonarchimedeanizations}.} two versions of prismatizations over $T/\mathbb{Q}_p$ or $T/\mathbb{F}_p((z))$. We start from the foundation in \cite{BSI}, \cite{BLI}, \cite{DI}, \cite{SchI}, \cite{ALBRCS}, \cite{TI}, \cite{TII}, \cite{TIII}, \cite{TVI}. We start from the corresponding formal scheme situation. We will start from a finite extension $T$ of $\mathbb{F}_p((z))$. We consider a general $F$ which is a $z$-adic formal scheme over $\mathcal{O}_T$. Attached to $F$ we have the corresponding primatic stack:
\begin{align}
\mathrm{\Gamma}_{\mathrm{prismatization},F}
\end{align}  
which is defined actually over rings where $z$ is nilpotent\footnote{Yes as in \cite{KLI}, \cite{KLII}, \cite{FS}, \cite{SchII} the underlying category of rings can be chosen to be the same.}. 
\begin{remark}
Here the definition goes completely the parallel as in \cite{BSI}, \cite{BLI}, \cite{DI}, \cite{SchI}, \cite{ALBRCS}, \cite{TI}, \cite{TII}, \cite{TIII}, \cite{TVI} where we just use the corresponding $\mathrm{WVector}_{\mathcal{O}_T}$ to define the family of Cartier sheaves. 
\end{remark}
We use the notation:
\begin{align}
\mathrm{C}_{z\mathrm{nil}} 
\end{align}
to be this category. Therefore the stack
\begin{align}
\mathrm{\Gamma}_{\mathrm{prismatization},F}
\end{align} 
is then defined over this category as the family of Cartier divisors (\textit{closed immersingly}) living in the corresponding generalized Witt vector of any ring $H$ in the under category:
\begin{align}
?\rightarrow \mathrm{WVector}_{\mathcal{O}_T}(H).
\end{align}
This stack is actually formal stack. One then have the corresponding quasicoherent sheaves over this stack as the corresponding $\infty$-category of all the quasicoherent modules over the structure sheaf:
\begin{align}
\mathrm{Quasicoherent}_{\mathrm{\Gamma}_{\mathrm{prismatization},F}}
\end{align}
which can be presented as the inverse limit of certain $\infty$-category of all the quasicoherent modules over a family of prisms in coherent way:
\begin{align}
\mathrm{Quasicoherent}_{\mathrm{\Gamma}_{\mathrm{prismatization},F}}=\projlim_i \mathrm{Quasicoherent}_{\mathrm{Spec}U_i}
\end{align}
where $U_i$ is the desired family of such sort of prisms\footnote{Here we require the prisms to take the form of a fixed map $?\rightarrow \mathrm{WVector}_{\mathcal{O}_T}(H)$. And as in \cite{BSI} the completeness is also required.}. This presentation can be directly use to construct the corresponding de Rham stacks for instance:
\begin{align}
\mathrm{\Gamma}_{\mathrm{prismatization},\mathrm{deRham},F}:=\injlim_i \mathrm{Spec}U_i[1/\eta]_{I_{U_i}}
\end{align}
with the corresponding quasicoherent sheaves category:
\begin{align}
\mathrm{Quasicoherent}_{\mathrm{\Gamma}_{\mathrm{prismatization},\mathrm{deRham},F}}=\projlim_i \mathrm{Quasicoherent}_{\mathrm{Spec}U_i[1/\eta]_{I_{U_i}}}.
\end{align}
One then have the analytification version of this stack:
\begin{align}
\mathrm{\Gamma}^\sharp_{\mathrm{prismatization},\mathrm{deRham},F}:=\injlim_i \mathrm{Spec}^\sharp U_i[1/\eta]_{I_{U_i}}
\end{align}
with the corresponding quasicoherent sheaves category:
\begin{align}
\mathrm{Quasicoherent}^\sharp_{\mathrm{\Gamma}_{\mathrm{prismatization},\mathrm{deRham},F}}=\projlim_i \mathrm{Quasicoherent}^\sharp_{\mathrm{Spec}U_i[1/\eta]_{I_{U_i}}}.
\end{align}
We then have the corresponding cristalline version as well. This presentation can be directly use to construct the corresponding cristalline stacks for instance:
\begin{align}
\mathrm{\Gamma}_{\mathrm{prismatization},\mathrm{cristalline},F}:=\injlim_i \mathrm{Spec}U_i[1/\eta]_{I_{U_i}}
\end{align}
with the corresponding quasicoherent sheaves category:
\begin{align}
\mathrm{Quasicoherent}_{\mathrm{\Gamma}_{\mathrm{prismatization},\mathrm{cristalline},F}}=\projlim_i \mathrm{Quasicoherent}_{\mathrm{Spec}U_i[1/\eta]_{I_{U_i}}}.
\end{align}
One then have the analytification version of this stack:
\begin{align}
\mathrm{\Gamma}^\sharp_{\mathrm{prismatization},\mathrm{cristalline},F}:=\injlim_i \mathrm{Spec}^\sharp U_i[1/\eta]_{I_{U_i}}
\end{align}
with the corresponding quasicoherent sheaves category:
\begin{align}
\mathrm{Quasicoherent}^\sharp_{\mathrm{\Gamma}_{\mathrm{prismatization},\mathrm{cristalline},F}}=\projlim_i \mathrm{Quasicoherent}^\sharp_{\mathrm{Spec}U_i[1/\eta]_{I_{U_i}}}.
\end{align} 

\begin{definition}
We now generlize the corresponding foundation above to \cite{BSB} in the following sense. First starting with the stack:
\begin{align}
\mathrm{\Gamma}_{\mathrm{prismatization},F}
\end{align}
we have the corresponding $p$-adic primitive element $h$ from \cite{BSB} which is usual the corresponding logrithmic with respect to the imperfect Robba ring (in variable $X$, then it will be the log of $1+X$) attached to the field $T$. We then add $h^{1/2}$ to the following map in the definition of the primatization:
\begin{align}
?\rightarrow \mathrm{WVector}_{\mathcal{O}_T}(H)[h^{1/2}].
\end{align}
Then we have the corresponding stack which we will denote it by:
\begin{align}
\mathrm{\Gamma}_{\mathrm{prismatization},F,2}.
\end{align}
Then the de Rham and cristalline situations are generalized accordingly. This stack is actually formal stack. One then have the corresponding quasicoherent sheaves over this stack as the corresponding $\infty$-category of all the quasicoherent modules over the structure sheaf:
\begin{align}
\mathrm{Quasicoherent}_{\mathrm{\Gamma}_{\mathrm{prismatization},F,2}}
\end{align}
which can be presented as the inverse limit of certain $\infty$-category of all the quasicoherent modules over a family of prisms in coherent way:
\begin{align}
\mathrm{Quasicoherent}_{\mathrm{\Gamma}_{\mathrm{prismatization},F,2}}=\projlim_i \mathrm{Quasicoherent}_{\mathrm{Spec}U_i[h^{1/2}]}
\end{align}
where $U_i$ is the desired family of such sort of prisms. This presentation can be directly use to construct the corresponding de Rham stacks for instance:
\begin{align}
\mathrm{\Gamma}_{\mathrm{prismatization},\mathrm{deRham},F,2}:=\injlim_i \mathrm{Spec}U_i[1/\eta]_{I_{U_i}}[h^{1/2}]
\end{align}
with the corresponding quasicoherent sheaves category:
\begin{align}
\mathrm{Quasicoherent}_{\mathrm{\Gamma}_{\mathrm{prismatization},\mathrm{deRham},F,2}}=\projlim_i \mathrm{Quasicoherent}_{\mathrm{Spec}U_i[1/\eta]_{I_{U_i}}[h^{1/2}]}.
\end{align}
One then have the analytification version of this stack:
\begin{align}
\mathrm{\Gamma}^\sharp_{\mathrm{prismatization},\mathrm{deRham},F,2}:=\injlim_i \mathrm{Spec}^\sharp U_i[1/p]_{I_{U_i}}[h^{1/2}]
\end{align}
with the corresponding quasicoherent sheaves category:
\begin{align}
\mathrm{Quasicoherent}^\sharp_{\mathrm{\Gamma}_{\mathrm{prismatization},\mathrm{deRham},F,2}}=\projlim_i \mathrm{Quasicoherent}^\sharp_{\mathrm{Spec}U_i[1/\eta]_{I_{U_i}}[h^{1/2}]}.
\end{align}
We then have the corresponding cristalline version as well. This presentation can be directly use to construct the corresponding cristalline stacks for instance:
\begin{align}
\mathrm{\Gamma}_{\mathrm{prismatization},\mathrm{cristalline},F,2}:=\injlim_i \mathrm{Spec}U_i[1/\eta][h^{1/2}]
\end{align}
with the corresponding quasicoherent sheaves category:
\begin{align}
\mathrm{Quasicoherent}_{\mathrm{\Gamma}_{\mathrm{prismatization},\mathrm{cristalline},F,2}}=\projlim_i \mathrm{Quasicoherent}_{\mathrm{Spec}U_i[1/\eta][h^{1/2}]}.
\end{align}
One then have the analytification version of this stack:
\begin{align}
\mathrm{\Gamma}^\sharp_{\mathrm{prismatization},\mathrm{cristalline},F,2}:=\injlim_i \mathrm{Spec}^\sharp U_i[1/\eta][h^{1/2}]
\end{align}
with the corresponding quasicoherent sheaves category:
\begin{align}
\mathrm{Quasicoherent}^\sharp_{\mathrm{\Gamma}_{\mathrm{prismatization},\mathrm{cristalline},F,2}}=\projlim_i \mathrm{Quasicoherent}^\sharp_{\mathrm{Spec}U_i[1/\eta][h^{1/2}]}.
\end{align} 
\end{definition}

\begin{theorem}
When $F$ is $\mathcal{O}_T$ we have a generalized condensed prismatization, which is well-defined.
\end{theorem}

\subsection{Perfectoid Picture}
\noindent We then consider the contact with the work of Kedlaya-Liu and Scholze, where untilts of perfectoids are parametrized by certain similar stacks, i.e. the Fargues-Fontaine stacks, from \cite{KLI}, \cite{KLII}, \cite{SchII}, \cite{SchIII}, \cite{SchIV}. We now consider any small $v$-stack $S$ over $\mathrm{Spd}\mathcal{O}_T$. For any local chart ${F}$ in the $v$-topology over this stack $S$, we can get the corresponding prismatization directly in a transparent way, namely the corresponding prismatization is just the corresponding space of the Witt vector attached to $F$. However this is a prismatic method. Attached to $F$ we have the corresponding primatic stack:
\begin{align}
\mathrm{\Gamma}_{\mathrm{prismatization},F}
\end{align}  
which is defined actually over rings where $z$ is nilpotent. 

\begin{remark}
Here the definition goes completely the parallel as in \cite{BSI}, \cite{BLI}, \cite{DI}, \cite{SchI}, \cite{ALBRCS}, \cite{TI}, \cite{TII}, \cite{TIII}, \cite{TVI} where we just use the corresponding $\mathrm{WVector}_{\mathcal{O}_T}$ to define the family of Cartier sheaves. 
\end{remark}

We use the notation:
\begin{align}
\mathrm{C}_{z\mathrm{nil}} 
\end{align}
to be this category. Therefore the stack
\begin{align}
\mathrm{\Gamma}_{\mathrm{prismatization},F}
\end{align} 
is then defined over this category as the family of Cartier divisors (\textit{closed immersingly}) living in the corresponding generalized Witt vector of any ring $H$ in the under category:
\begin{align}
?\rightarrow \mathrm{WVector}_{\mathcal{O}_T}(H).
\end{align}
This stack is actually formal stack. One then have the corresponding quasicoherent sheaves over this stack as the corresponding $\infty$-category of all the quasicoherent modules over the structure sheaf:
\begin{align}
\mathrm{Quasicoherent}_{\mathrm{\Gamma}_{\mathrm{prismatization},F}}
\end{align}
which can be presented as the inverse limit of certain $\infty$-category of all the quasicoherent modules over a family of prisms in coherent way:
\begin{align}
\mathrm{Quasicoherent}_{\mathrm{\Gamma}_{\mathrm{prismatization},F}}=\projlim_i \mathrm{Quasicoherent}_{\mathrm{Spec}U_i}
\end{align}
where $U_i$ is the desired family of such sort of prisms.\footnote{Here we require the prisms to take the form of a fixed map $?\rightarrow \mathrm{WVector}_{\mathcal{O}_T}(H)$. And as in \cite{BSI} the completeness is also required. Completeness is defined to be the same as in \cite{BSI} where we just replace $p$ by $z$.} This presentation can be directly use to construct the corresponding de Rham stacks for instance:
\begin{align}
\mathrm{\Gamma}_{\mathrm{prismatization},\mathrm{deRham},F}:=\injlim_i \mathrm{Spec}U_i[1/\eta]_{I_{U_i}}
\end{align}
with the corresponding quasicoherent sheaves category:
\begin{align}
\mathrm{Quasicoherent}_{\mathrm{\Gamma}_{\mathrm{prismatization},\mathrm{deRham},F}}=\projlim_i \mathrm{Quasicoherent}_{\mathrm{Spec}U_i[1/\eta]_{I_{U_i}}}.
\end{align}
One then have the analytification version of this stack:
\begin{align}
\mathrm{\Gamma}^\sharp_{\mathrm{prismatization},\mathrm{deRham},F}:=\injlim_i \mathrm{Spec}^\sharp U_i[1/\eta]_{I_{U_i}}
\end{align}
with the corresponding quasicoherent sheaves category:
\begin{align}
\mathrm{Quasicoherent}^\sharp_{\mathrm{\Gamma}_{\mathrm{prismatization},\mathrm{deRham},F}}=\projlim_i \mathrm{Quasicoherent}^\sharp_{\mathrm{Spec}U_i[1/\eta]_{I_{U_i}}}.
\end{align}
We then have the corresponding cristalline version as well. This presentation can be directly use to construct the corresponding cristalline stacks for instance:
\begin{align}
\mathrm{\Gamma}_{\mathrm{prismatization},\mathrm{cristalline},F}:=\injlim_i \mathrm{Spec}U_i[1/\eta]_{I_{U_i}}
\end{align}
with the corresponding quasicoherent sheaves category:
\begin{align}
\mathrm{Quasicoherent}_{\mathrm{\Gamma}_{\mathrm{prismatization},\mathrm{cristalline},F}}=\projlim_i \mathrm{Quasicoherent}_{\mathrm{Spec}U_i[1/\eta]_{I_{U_i}}}.
\end{align}
One then have the analytification version of this stack:
\begin{align}
\mathrm{\Gamma}^\sharp_{\mathrm{prismatization},\mathrm{cristalline},F}:=\injlim_i \mathrm{Spec}^\sharp U_i[1/\eta]_{I_{U_i}}
\end{align}
with the corresponding quasicoherent sheaves category:
\begin{align}
\mathrm{Quasicoherent}^\sharp_{\mathrm{\Gamma}_{\mathrm{prismatization},\mathrm{cristalline},F}}=\projlim_i \mathrm{Quasicoherent}^\sharp_{\mathrm{Spec}U_i[1/\eta]_{I_{U_i}}}.
\end{align} 

\begin{definition}
We now generlize the corresponding foundation above to \cite{BSB} in the following sense. First starting with the stack:
\begin{align}
\mathrm{\Gamma}_{\mathrm{prismatization},F}
\end{align}
we have the corresponding $p$-adic primitive element $h$ from \cite{BSB} which is usual the corresponding logrithmic with respect to the imperfect Robba ring (in variable $X$, then it will be the log of $1+X$) attached to the field $T$. We then add $h^{1/2}$ to the following map in the definition of the primatization:
\begin{align}
?\rightarrow \mathrm{WVector}_{\mathcal{O}_T}(H)[h^{1/2}].
\end{align}
Then we have the corresponding stack which we will denote it by:
\begin{align}
\mathrm{\Gamma}_{\mathrm{prismatization},F,2}.
\end{align}
Then the de Rham and cristalline situations are generalized accordingly. This stack is actually formal stack. One then have the corresponding quasicoherent sheaves over this stack as the corresponding $\infty$-category of all the quasicoherent modules over the structure sheaf:
\begin{align}
\mathrm{Quasicoherent}_{\mathrm{\Gamma}_{\mathrm{prismatization},F,2}}
\end{align}
which can be presented as the inverse limit of certain $\infty$-category of all the quasicoherent modules over a family of prisms in coherent way:
\begin{align}
\mathrm{Quasicoherent}_{\mathrm{\Gamma}_{\mathrm{prismatization},F,2}}=\projlim_i \mathrm{Quasicoherent}_{\mathrm{Spec}U_i[h^{1/2}]}
\end{align}
where $U_i$ is the desired family of such sort of prisms. \footnote{Here we require the prisms to take the form of a fixed map $?\rightarrow \mathrm{WVector}_{\mathcal{O}_T}(H)$. And as in \cite{BSI} the completeness is also required. Completeness is defined to be the same as in \cite{BSI} where we just replace $p$ by $z$.} This presentation can be directly use to construct the corresponding de Rham stacks for instance:
\begin{align}
\mathrm{\Gamma}_{\mathrm{prismatization},\mathrm{deRham},F,2}:=\injlim_i \mathrm{Spec}U_i[1/\eta]_{I_{U_i}}[h^{1/2}]
\end{align}
with the corresponding quasicoherent sheaves category:
\begin{align}
\mathrm{Quasicoherent}_{\mathrm{\Gamma}_{\mathrm{prismatization},\mathrm{deRham},F,2}}=\projlim_i \mathrm{Quasicoherent}_{\mathrm{Spec}U_i[1/\eta]_{I_{U_i}}[h^{1/2}]}.
\end{align}
One then have the analytification version of this stack:
\begin{align}
\mathrm{\Gamma}^\sharp_{\mathrm{prismatization},\mathrm{deRham},F,2}:=\injlim_i \mathrm{Spec}^\sharp U_i[1/\eta]_{I_{U_i}}[h^{1/2}]
\end{align}
with the corresponding quasicoherent sheaves category:
\begin{align}
\mathrm{Quasicoherent}^\sharp_{\mathrm{\Gamma}_{\mathrm{prismatization},\mathrm{deRham},F,2}}=\projlim_i \mathrm{Quasicoherent}^\sharp_{\mathrm{Spec}U_i[1/\eta]_{I_{U_i}}[h^{1/2}]}.
\end{align}
We then have the corresponding cristalline version as well. This presentation can be directly use to construct the corresponding cristalline stacks for instance:
\begin{align}
\mathrm{\Gamma}_{\mathrm{prismatization},\mathrm{cristalline},F,2}:=\injlim_i \mathrm{Spec}U_i[1/\eta][h^{1/2}]
\end{align}
with the corresponding quasicoherent sheaves category:
\begin{align}
\mathrm{Quasicoherent}_{\mathrm{\Gamma}_{\mathrm{prismatization},\mathrm{cristalline},F,2}}=\projlim_i \mathrm{Quasicoherent}_{\mathrm{Spec}U_i[1/\eta][h^{1/2}]}.
\end{align}
One then have the analytification version of this stack:
\begin{align}
\mathrm{\Gamma}^\sharp_{\mathrm{prismatization},\mathrm{cristalline},F,2}:=\injlim_i \mathrm{Spec}^\sharp U_i[1/\eta][h^{1/2}]
\end{align}
with the corresponding quasicoherent sheaves category:
\begin{align}
\mathrm{Quasicoherent}^\sharp_{\mathrm{\Gamma}_{\mathrm{prismatization},\mathrm{cristalline},F,2}}=\projlim_i \mathrm{Quasicoherent}^\sharp_{\mathrm{Spec}U_i[1/\eta][h^{1/2}]}.
\end{align} 
\end{definition}

\indent Then we vary $F$ in the $v$-site for $S$ we then have the corresponding the condensed prismatization of the $v$-stack $S$ in the mixed-parity generalization fashion.

\newpage
\section{Applications}

\subsection{Rigid Analytic Spaces}

\begin{definition}
We now construct the condensed motives for rigid analytic spave $R$ over $T$ through the perfectoid consideration above from $v$-stacks. One can certainly go along \cite{SchI} and \cite{ALBRCS} to glue along the formal model locally to reach the whole motives for the spaces. In the perfectoid setting we regard $R$ in genral as a corresponding $v$-stacks over $\mathrm{Spd}T$. Then we have the following $\infty$-categories in the generalization we considered here:
\begin{align}
&\mathrm{Quasicoherent}^\sharp_{\mathrm{\Gamma}_{\mathrm{prismatization},R,2}}=\lim_{F}\projlim_i \mathrm{Quasicoherent}^\sharp_{\mathrm{Spec}U_i[h^{1/2}]},\\
&\mathrm{Quasicoherent}^\sharp_{\mathrm{\Gamma}_{\mathrm{prismatization},\mathrm{deRham},R,2}}=\lim_{F}\projlim_i \mathrm{Quasicoherent}^\sharp_{\mathrm{Spec}U_i[1/\eta]_{I_{U_i}}[h^{1/2}]},\\
&\mathrm{Quasicoherent}^\sharp_{\mathrm{\Gamma}_{\mathrm{prismatization},\mathrm{cristalline},R,2}}=\lim_{F}\projlim_i \mathrm{Quasicoherent}^\sharp_{\mathrm{Spec}U_i[1/\eta][h^{1/2}]}.
\end{align}
Here $\eta$ is either $p$ or $z$. Here in the $z$-adic setting the picture will be completely $z$-adic following \cite{SchIV}, \cite{KLI}, \cite{KLII}. The $v$-site relative consideration in \cite{SchIV}, \cite{KLI}, \cite{KLII} enriced over the pro-\'etale consideration actually relates to the above picture by taking the projective limit onto the corresponding perfectoid prisms in the limit sequence. This also generalizes \cite{TIV} significantly. 
\end{definition}

In \cite{TIV} we discuss only the $p$-adic situation with generalization from \cite{BSB}. Go along those consideration in \cite{TIV} we have a completely parallel picture in $z$-adic setting with the $z$-adic $h$, where $h$ is defined parallel to the $p$-adic setting. 

\begin{definition}
We now construct the condensed motives for rigid analytic spave $R$ over $T$ through the perfectoid consideration above from $v$-stacks. One can certainly go along \cite{SchI} and \cite{ALBRCS} to glue along the formal model locally to reach the whole motives for the spaces. In the perfectoid setting we regard $R$ in genral as a corresponding $v$-stacks over $\mathrm{Spd}T$. Then we have the following $\infty$-categories in the generalization we considered here:
\begin{align}
\mathrm{Quasicoherent}^\sharp_{\mathrm{\Gamma}_{\mathrm{prismatization},R,2}}=\lim_{F}\projlim_i \mathrm{Quasicoherent}^\sharp_{\mathrm{Spec}U_i[h^{1/2}]},
\end{align}
\begin{align}
\mathrm{Quasicoherent}^\sharp_{\mathrm{\Gamma}_{\mathrm{prismatization},\mathrm{deRham},R,2}}=\lim_{F}\projlim_i \mathrm{Quasicoherent}^\sharp_{\mathrm{Spec}U_i[1/\eta]_{I_{U_i}}[h^{1/2}]},
\end{align}
\begin{align}
\mathrm{Quasicoherent}^\sharp_{\mathrm{\Gamma}_{\mathrm{prismatization},\mathrm{cristalline},R,2}}=\lim_{F}\projlim_i \mathrm{Quasicoherent}^\sharp_{\mathrm{Spec}U_i[1/\eta][h^{1/2}]}.
\end{align}
Here $\eta$ is either $p$ or $z$. Here in the $z$-adic setting the picture will be completely $z$-adic following \cite{SchIV}, \cite{KLI}, \cite{KLII}. The $v$-site relative consideration in \cite{SchIV}, \cite{KLI}, \cite{KLII} enriced over the pro-\'etale consideration actually relates to the above picture by taking the projective limit onto the corresponding perfectoid prisms in the limit sequence. This also generalizes \cite{TIV} significantly. We then have the following Fontaine style functors:
\begin{align}
\mathrm{Quasicoherent}^\sharp_{\mathrm{\Gamma}_{\mathrm{prismatization},R,2}}=\lim_{F}\projlim_i \mathrm{Quasicoherent}^\sharp_{\mathrm{Spec}U_i[h^{1/2}]}\\\rightarrow \mathrm{Quasicoherent}^\sharp_{\mathrm{\Gamma}_{\mathrm{prismatization},T,2}},
\end{align}
\begin{align}
\mathrm{Quasicoherent}^\sharp_{\mathrm{\Gamma}_{\mathrm{prismatization},\mathrm{deRham},R,2}}=\lim_{F}\projlim_i \mathrm{Quasicoherent}^\sharp_{\mathrm{Spec}U_i[1/\eta]_{I_{U_i}}[h^{1/2}]}\\
\rightarrow \mathrm{Quasicoherent}^\sharp_{\mathrm{\Gamma}_{\mathrm{prismatization},T,2}},
\end{align}
\begin{align}
\mathrm{Quasicoherent}^\sharp_{\mathrm{\Gamma}_{\mathrm{prismatization},\mathrm{cristalline},R,2}}=\lim_{F}\projlim_i \mathrm{Quasicoherent}^\sharp_{\mathrm{Spec}U_i[1/\eta][h^{1/2}]}\\
\rightarrow \mathrm{Quasicoherent}^\sharp_{\mathrm{\Gamma}_{\mathrm{prismatization},T,2}},
\end{align}
through the push forward along the structure morphism for $R$ to $\mathrm{Spd}T$. For any bundle $B$ (we \textbf{allow} infinite rank bundle here beyond the vector bundle situation, and we do not require the existence of Frobenius structure while one can definitely add the shtukas structures) over the Robba sheaf $\Pi_{\mathrm{perf}, \blacksquare,v,2}$ (adding the element $h^{1/2}$ to the usual Robba sheaf in the perfect setting) from \cite{KLI}, \cite{KLII}, we can define the corresponding functor of the generalized de Rham functor as:
\begin{align}
(\mathrm{Quasicoherent}^\sharp_{\mathrm{\Gamma}_{\mathrm{prismatization},\mathrm{deRham},R,2}}=\lim_{F}\projlim_i \mathrm{Quasicoherent}^\sharp_{\mathrm{Spec}U_i[1/\eta]_{I_{U_i}}[h^{1/2}]}\\
\rightarrow \mathrm{Quasicoherent}^\sharp_{\mathrm{\Gamma}_{\mathrm{prismatization},T,2}})_\sharp(B\otimes \mathcal{O}_{\mathrm{\Gamma}^\sharp_{\mathrm{prismatization},\mathrm{deRham},F,2}}).
\end{align}
We call this sheaf de Rham in the mixed parity situation if an isomorphism retains after the base change back over to $\mathcal{O}_{\mathrm{\Gamma}^\sharp_{\mathrm{prismatization},\mathrm{deRham},F,2}}$. One then defines  the corresponding cristalline sheaves over the Robba rings $\Pi_{\mathrm{perf}, \blacksquare,v,2}$ in the same fashion.
\end{definition}

We remind the readers that this is not just $z$-adic generalization from \cite{TIV}, it is a prismatization generalization, where we will see the motivic structures richer than the context of \cite{TIV} based on the pro-\'etale cohomology.

\begin{theorem}
We have a well-defined functor from solid quasicoherent sheaves over mixed-parity pre-Fargues-Fontaine stack attached to $R$ (taking the analytic stack in the sense of \cite{CS3} of $\Pi_{\mathrm{perf},\blacksquare,v,2}$ without taking the Frobenius quotients) to solid quasicoherent sheaves over the de Rham stack we considered in the mixed-parity setting. We have a well-defined functor from solid quasicoherent sheaves over mixed-parity pre-Fargues-Fontaine stack attached to $R$ (taking the analytic stack in the sense of \cite{CS3} of $\Pi_{\mathrm{perf},\blacksquare,v,2}$ without taking the Frobenius quotients) to solid quasicoherent sheaves over the cristalline stack we considered in the mixed-parity setting. 
\end{theorem}

\begin{theorem}
We have a well-defined functor from solid quasicoherent sheaves over mixed-parity Fargues-Fontaine stack attached to $R$ (taking the analytic stack in the sense of \cite{CS3} of $\Pi_{\mathrm{perf},\blacksquare,v,2}$) to solid quasicoherent sheaves over the de Rham stack we considered in the mixed-parity setting. We have a well-defined functor from solid quasicoherent sheaves over mixed-parity Fargues-Fontaine stack attached to $R$ (taking the analytic stack in the sense of \cite{CS3} of $\Pi_{\mathrm{perf},\blacksquare,v,2}$) to solid quasicoherent sheaves over the cristalline stack we considered in the mixed-parity setting.  
\end{theorem}

\subsection{Generalized Langlands Program}

\noindent In this section we study the applications of the generalization we discussed above on the prismatization. Here we following \cite{FS} to make certain $p$-adic cohomologicalization to the picture in the $\ell$-adic setting in the \cite{FS}. This will also in some sense generalize the following papers: \cite{TIV}, \cite{TV}, \cite{TVI} in the sense that we will also consider mixed-parity generalization of the motives we are considering while the underlying $v$-stacks are also generalized from \cite{FS} by directly considering the Fargues-Fontaine stacks constructed using Robba rings where we allow the square root of $h$ in the previos discussion in the previous sections. Recall that for such FF stack:
\begin{align}
\mathrm{\Gamma}_{\mathrm{FF},h^{1/2}}(.)
\end{align}
we have the associated $G$-bundle stack in the moduli sense:
\begin{align}
\mathrm{\Gamma}_{\mathrm{Bun},h^{1/2}}(.)^G
\end{align}
where $G/T$ over $T$ have two different form in different characteristic. In our current consideration we can then apply the consideration above to achive the corresponding motives over this stack by using basis like the $F$ in the previous sections.

\begin{theorem}
Over the stack:
\begin{align}
\mathrm{\Gamma}_{\mathrm{Bun},h^{1/2}}(.)^G
\end{align}
we have the quasicoherent motives over the stackification by using the two different forms of the corresponding prismatizations. In both situations we have the corresponding generalized condensation of the corresponding prismatization for $v$-stacks:
\begin{align}
\mathrm{Quasicoherent}^\sharp_{\mathrm{\Gamma}_{\mathrm{prismatization},\mathrm{\Gamma}_{\mathrm{Bun},h^{1/2}}(.)^G,2}}=\lim_{F}\projlim_i \mathrm{Quasicoherent}^\sharp_{\mathrm{Spec}U_i[h^{1/2}]},
\end{align}
\begin{align}
\mathrm{Quasicoherent}^\sharp_{\mathrm{\Gamma}_{\mathrm{prismatization},\mathrm{deRham},\mathrm{\Gamma}_{\mathrm{Bun},h^{1/2}}(.)^G,2}}=\lim_{F}\projlim_i \mathrm{Quasicoherent}^\sharp_{\mathrm{Spec}U_i[1/\eta]_{I_{U_i}}[h^{1/2}]},
\end{align}
\begin{align}
\mathrm{Quasicoherent}^\sharp_{\mathrm{\Gamma}_{\mathrm{prismatization},\mathrm{cristalline},\mathrm{\Gamma}_{\mathrm{Bun},h^{1/2}}(.)^G,2}}=\lim_{F}\projlim_i \mathrm{Quasicoherent}^\sharp_{\mathrm{Spec}U_i[1/\eta][h^{1/2}]}.
\end{align}
Here $\eta$ is either $p$ or $z$. In the $p$-adic situation, suppose we consider the following two $\infty$-categories:
\begin{align}
\mathrm{Quasicoherent}^\sharp_{\mathrm{\Gamma}_{\mathrm{prismatization},\mathrm{deRham},\mathrm{\Gamma}_{\mathrm{Bun},h^{1/2}}(.)^G,2}}=\lim_{F}\projlim_i \mathrm{Quasicoherent}^\sharp_{\mathrm{Spec}U_i[1/\eta]_{I_{U_i}}[h^{1/2}]},
\end{align}
\begin{align}
\mathrm{Quasicoherent}^\sharp_{\mathrm{\Gamma}_{\mathrm{prismatization},\mathrm{cristalline},\mathrm{\Gamma}_{\mathrm{Bun},h^{1/2}}(.)^G,2}}=\lim_{F}\projlim_i \mathrm{Quasicoherent}^\sharp_{\mathrm{Spec}U_i[1/\eta][h^{1/2}]}.
\end{align}
And we assume that we tensor with the large coefficient field $\overline{\mathbb{Q}}_p$:
\begin{align}
\mathrm{Quasicoherent}^\sharp_{\mathrm{\Gamma}_{\mathrm{prismatization},\mathrm{deRham},\mathrm{\Gamma}_{\mathrm{Bun},h^{1/2}}(.)^G,2,\overline{\mathbb{Q}}_p}}=\lim_{F}\projlim_i \mathrm{Quasicoherent}^\sharp_{\mathrm{Spec}U_i[1/\eta]_{I_{U_i}}[h^{1/2}]}\otimes^\sharp \overline{\mathbb{Q}}_p,
\end{align}
\begin{align}
\mathrm{Quasicoherent}^\sharp_{\mathrm{\Gamma}_{\mathrm{prismatization},\mathrm{cristalline},\mathrm{\Gamma}_{\mathrm{Bun},h^{1/2}}(.)^G,2,\overline{\mathbb{Q}}_p}}=\lim_{F}\projlim_i \mathrm{Quasicoherent}^\sharp_{\mathrm{Spec}U_i[1/\eta][h^{1/2}]}\otimes^\sharp \overline{\mathbb{Q}}_p.
\end{align}
Then we can realize a generalized Langlands parametrization with $\overline{\mathbb{Q}}_p$ coefficient with operation over the categories from the Weil groups in the mixed-parity setting after \cite{FS}, \cite{TIV}, \cite{TV}, \cite{TVI}.
\end{theorem}

\begin{proof}
See the proof in \cite[Chapter VIII, IX]{FS}, \cite{VL}, \cite{TIV}, \cite{TV}, \cite{TVI}. Here the key consideration comes from reaching a corresponding correspondence in between the $p$-adic coefficients and the $\ell$-adic coefficients through the algebraic isomorphism:
\begin{align}
\overline{\mathbb{Q}}_p \overset{\sim}{\longrightarrow} \overline{\mathbb{Q}}_\ell.
\end{align}
From this isomorphism we start from the corresponding representation of $p$-adic points ($\overline{\mathbb{Q}}_p$-values) of the full Langlands dual group (with the action from the Weil group), then we can end up with the corresponding representation of the corresponding representation of $\ell$-adic points ($\overline{\mathbb{Q}}_\ell$-values) of the full Langlands dual group (with the action from the Weil group), which produces a motivic complex through the Hecke operator, which then produces a motivic complex with $\overline{\mathbb{Q}}_p$ coefficients. This will then be sent to the categories in our generalized setting by taking the condensed product with the prismatization directly.
\end{proof}

\begin{definition}
We use the following to define the solid geometric generalized Banach space representation of $G(T)$:
\begin{align}
\mathrm{Quasicoherent}^\sharp_{\mathrm{\Gamma}_{\mathrm{prismatization},\mathrm{\Gamma}_{\mathrm{Bun},h^{1/2}}(.)^G,2}}=\lim_{F}\projlim_i \mathrm{Quasicoherent}^\sharp_{\mathrm{Spec}U_i[h^{1/2}]},
\end{align}
\begin{align}
\mathrm{Quasicoherent}^\sharp_{\mathrm{\Gamma}_{\mathrm{prismatization},\mathrm{deRham},\mathrm{\Gamma}_{\mathrm{Bun},h^{1/2}}(.)^G,2}}=\lim_{F}\projlim_i \mathrm{Quasicoherent}^\sharp_{\mathrm{Spec}U_i[1/\eta]_{I_{U_i}}[h^{1/2}]},
\end{align}
\begin{align}
\mathrm{Quasicoherent}^\sharp_{\mathrm{\Gamma}_{\mathrm{prismatization},\mathrm{cristalline},\mathrm{\Gamma}_{\mathrm{Bun},h^{1/2}}(.)^G,2}}=\lim_{F}\projlim_i \mathrm{Quasicoherent}^\sharp_{\mathrm{Spec}U_i[1/\eta][h^{1/2}]}.
\end{align}
\end{definition}

\begin{remark}
In some situation, the coverings of the Galois groups after \cite{BSB} will be trivial such as splitting as a pure product with the group scheme $\mu_n$, but this does not matter since on the other hand the Hodge structure is already generalized by adding the roots of the $p$-adic or $z$-adic '2$\pi$i'.
\end{remark}

\newpage
\subsection*{Acknowledgements}
Being far beyond the original (with further originality from previous work of Breuil and Schneider) perspective from Professor Sorensen, we express our thankfulness to Professor Sorensen for sharing understanding on the Breuil-Schneider conjecture and the philosophy. We also express our thankfulness to Professor Kedlaya for sharing understanding on the philosophy of Witt vectorial approach to the $p$-adic motives. This paper is profoundly influenced by previous philosophicalization from Professor Kedlaya on Witt vectors and the applications to motives.

\end{document}